\documentclass{article}

\title{\bf Toroidal Embeddings of Right Groups}

\author{Kolja Knauer, Ulrich Knauer\\
       \small{\tt knauer@\{math.tu-berlin.de, uni-oldenburg.de\}}}

\usepackage{pstricks,psfrag}
\usepackage[exercise]{amsthm}
\usepackage{amsfonts}
\usepackage{amssymb}
\usepackage{amsmath}
\usepackage{graphicx}
\usepackage{indentfirst}
\usepackage{multicol}
\usepackage{graphics}

\newtheorem{theo}{Theorem}[section]

\newtheorem{prop}[theo]{Proposition}
\newtheorem{lemm}[theo]{Lemma}

\newtheorem{resu}[theo]{Result}
\newtheorem{rema}[theo]{Remark}
\newtheorem{exam}[theo]{Example}

\theoremstyle{remark}

\theoremstyle{plain}

\begin{document}
\maketitle
\begin{abstract}
In this note we study embeddings of Cayley graphs of right groups on
surfaces. We characterize those right groups which have a toroidal
but no planar Cayley graph, such that the generating system of the
right group has a minimal generating system of the group as a
factor.
\end{abstract}

Keywords: Cayley graph, right group, planar, toroidal, embedding

Mathematics Subject Classification: 05C10, 05C25, 20M30, 57M15

\section{Preliminaries}
A graph is said to be \textit{(2-cell-)embedded} in a surface $M$ if it is
``drawn" in $M$ such that edges intersect only at their common
vertices and deleting the graph from $M$ yields a disjoint union of disks.
A graph is said to be \textit{planar} if it can be
embedded in the plane. By the \textit{genus} of a graph $X$
we mean the minimum genus among all surfaces in which $X$ can be
embedded. So if $X$ is planar then the genus of $X$ is zero. If a
non-planar graph can be embedded on the torus, that is on the
orientable surface of genus 1, it is called \textit{toroidal}. A
graph is said to be \textit{outer planar} if it has an embedding such that one face is incident to every vertex.

 It is known that each group can be defined in terms of generators
and relations, and that corresponding to each such (non-unique)
presentation there is a unique graph, called the Cayley
graph of the presentation. A ``drawing" of this graph gives a
``picture" of the group from which certain properties of the group
can be determined. The same principle can be used for other
algebraic systems. So algebraic systems with a given system of
generators will be called \textit{planar} or \textit{toroidal} if
the respective Cayley graphs can be embedded on the plane or on the
torus.

Finite planar groups have been cataloged by Maschke~\cite{M}. On
the basis of Maschke's Theorem, in this work we investigate
embeddings of certain completely regular semigroups (unions of
groups), namely of right groups. This is a continuation of the
investigations from~\cite{Xia} where Clifford semigroups were in
focus. Here our attention is restricted to a special class of
presentations of right groups for which we classify the toroidal
right groups. Note that this generally only gives upper bounds on
the genus of right groups. The full determination of the genus will
be studied in a subsequent paper~\cite{kk}.

We use $K_n$ for the complete graph on $n$ vertices, $C_n$ for the
cycle on $n$ vertices, and $K_{n,n}$ for the respective complete
bipartite graph. We denote the cyclic group of order $n$ by
$\mathbb{Z}_n=\{0,\dots,n-1\}$, and the dihedral, symmetric and
alternating groups by $D_n$, $S_n$ and $A_n$, respectively.

We recall that a \textit{right group} is a semigroup of the form
$G\times R_r$ where $G$ is a group and $R_r$ is a right zero
semigroup, i. e., $R_r=\{r_1,\dots,r_r\}$ with the multiplication
$r_ir_j=r_j$ for $r_i,r_j\in R_r$.

Every semigroup presentation is associated with a
\textit{Cayley color graph}: the vertices correspond to the elements
of the semigroup; next, imagine the generators of the semigroup to
be associated with distinct colors. If vertices $v_{1}$ and $v_{2}$
correspond to semigroup elements $s_{1}$ and $s_{2}$ respectively,
then there is a directed edge (of the color of the generator $e$)
from $v_{1}$ to $v_{2}$ if and only if $s_{1}e=s_{2}$. It is also
possible to construct a Cayley color  graph by action from the left.
It is clear that for semigroups the structure of this graph may
change heavily, when changing the side of the action.

In this note we consider the graph obtained from the Cayley
color graph by suppressing all edge directions and all edge colors,
deleting loops and multiple edges, that is, the uncolored
Cayley graph. It is clear that in passing from the Cayley color graph
to the corresponding uncolored graph algebraical information is lost but
the genus is not changed. We call this graph
\textit{Cayley graph} and denote it by ${\it Cay}(S,C)$ for the semigroup
$S$ with the set of generators $C\subseteq S$.

The reader is referred to \cite{GT}, \cite{IK}, \cite{Kilp},
 \cite{Pe},  \cite{White} and \cite{Xia} for the terminology and notations which are
not given in this paper.

We need the following results.
\begin{resu}\label{Euler}{\em (Euler, Poincar{\'e} 1758)}
A finite graph with $n$ vertices, $m$ edges, which is $2$-cell
embedded on an orientable surface
 $M$ of genus $g$ with $f$ faces fulfills the Euler-Poincar{\'e} formula: $n-m+f=2-2g$.
\end{resu}

\begin{resu}\label{genuslemm1}{\em (Maschke 1896)} The finite group
$G$ is planar if and only if $G=G_{1} \times G_{2}$, where
$G_{1}=\mathbb{Z}_{1}$ or $\mathbb{Z}_{2}$ and
$G_{2}=\mathbb{Z}_{n}$, $D_{n}$, $S_{4}$, $A_{4}$ or $A_{5}$.
\end{resu}

\begin{rema} \label{generatorrema}\rm{It is clear that planarity depends on the set of
generators $C$ chosen for the Cayley graph. For example ${\it
Cay}(\mathbb{Z}_6,\{1\})=C_6$ and also ${\it
Cay}(\mathbb{Z}_6,\{2,3\})$ which is the box product $ C_3\Box K_2$
is planar, but ${\it Cay}(\mathbb{Z}_6,\{1,2,3\})=K_6$ is not. For
the planar groups $D_{n}$, $S_{4}$, $A_{4}$ or $A_{5}$ we get
various Archemedian solids as Cayley graph representations, with two
or three generators \cite{www}}.
\end{rema}

\begin{resu}\label{genuslemm2}{\em (Kuratowski 1930)}
A finite graph is planar if and only if it does not contain a
subgraph that is a subdivision of $K_{5}$ or $K_{3,3}$.
\end{resu}

\begin{resu}\label{genuslemm3}{\em (Chartrand, Harary 1967)}
A finite graph is outer planar if and only if it does not contain a
subgraph that is a subdivision of $K_4$ or $K_{2,3}$.
\end{resu}

\section{The $Cay$-functor and right groups}
For most of the considerations we can use the following two results
which we take from \cite{UX}. However, as far as we know, there do
not exist general formulas which relate the genus of a cross product
or a lexicographic product of two graphs to the genera of the
factors, compare for example \cite{GT}, \cite{IK} or \cite{White}.
Some of the difficulties with respect to the lexicographic product
can be seen in Example \ref{prop:D3}. We denote by $\times$ the
\textit{cross product} for graphs and also the direct product for
semigroups and sets. By $X[Y]$ we denote the \textit{lexicographic
product} of the graph $X$ with the graph $Y$.

\begin{prop}\label{theo2}
For semigroups $S$ and $T$ with subsets $C$ and $D$, respectively, we
have ${\it Cay}(S\times T, C\times D)={\it Cay}(S,C)\times
{\it Cay}(T,D)$.

\end{prop}

Note that if in the above formula the semigroup $T$ is $R_r$ its graph
${\it Cay}(R_r,R_r)$ has to be considered as $K_r^{(r)}$, i. e. the
complete graph with $r$ loops.

\begin{prop}\label{theo2a} Let $S$ be a monoid with identity $1_S$, $T$ a semigroup, $C$ and $D$
subsets of $S$ and $T$ respectively. Then $${\it Cay}(S\times T,
(C\times T) \cup (\{1_S\}\times D))={\it Cay}(S,C)[{\it Cay}(T,D)]$$
if and only if $tT=T$ for any $t\in T$, that is if and only if $T$
is a right
 group.
\end{prop}

\begin{rema} \rm{A formal description of the relation between graphs and subgraphs which are subdivisions
with the help of the $Cay$-functor on semigroups with generators
seems to be difficult. In ${\it Cay}(\mathbb{Z}_6,\{1\})$ we find a
subdivision of $K_3$ corresponding to ${\it Cay}(\{0, 2, 4\},
\{2\})$, as a subgraph. But subdivision is not a categorical
concept. And there is no inclusion
  between
$\{0, 2, 4\}\times \{2\}$ and $\mathbb{Z}_6\times \{1\}$. }
\end{rema}

\section{The embeddings}

Now we determine the minimal genus among the Cayley graphs ${\it
Cay}(G\times R_r, C\times R_r)$ taken over all minimum generating
set $C$ of the group $G$. We do not claim that an embedding of this
graph gives the (minimal) genus of the right group considered.
Generally $G\times R_r$ may have a generating system $C'\neq C\times R_r$ which yields a Cayley graph with fewer edges and
consequently tends to have a smaller genus. A straight-forward
calculation yields the following lemma. Note that the first equality
can also be obtained by applying Proposition \ref{theo2a} in the
form ${\it Cay}(G\times R_r, (C\times R_r) \cup (\{1_G\}\times
\emptyset))={\it Cay}(G,C)[{\it Cay}(R_r,\emptyset)]$.

\begin{lemm}
 Denote by ${\it Cay}(G,C)[\overline K_r]$
the lexicographic product of ${\it Cay}(G,C)$ with $r$ isolated
vertices. We have ${\it Cay}(G\times R_r, C\times R_r)={\it Cay}(G,C)[\overline K_r]$.
\end{lemm}

Note that this product can be seen as replacing every vertex of
${\it Cay}(G,C)$ by $r$ independent vertices and every edge by a
$K_{r,r}$. In particular $K_{k,k}[ \overline
K_r]=K_{kr,kr}$.

\begin{prop}
If ${\it Cay}(G,C)$ is not planar then ${\it Cay}(G\times R_r,
C\times R_r)$ cannot be embedded on the torus.
\end{prop}
\begin{proof}
Already $K_{3,3}[\overline K_{2}]\cong K_{6,6}$ has genus 4.
Moreover, the graph $K_5[\overline K_{2}]$  has 10 vertices and 40
edges. An embedding on the torus would have 30 faces by the formula
of Euler-Poincar{\'e}. Even if all faces were triangles in this
graph, this would require 45 edges. So the graphs are not toroidal.
\end{proof}

\begin{prop}
 If $r\geq 5$ then ${\it Cay}(G\times R_r, C\times R_r)$ cannot be embedded on the torus.
\end{prop}
\begin{proof}
 The resulting graph contains $K_{5,5}$ which has genus 3, compare~\cite{White}.
\end{proof}

\begin{prop}\label{prop:K22}
 If ${\it Cay}(G,C)$ contains a $K_{2,2}$ subdivision and $r\geq 3$ then ${\it Cay}(G\times R_r, C\times R_r)$ cannot be
 embedded on the torus.
\end{prop}
\begin{proof}
 The resulting graph contains $K_{6,6}$ which has genus 4, compare~\cite{White}.
\end{proof}

Hence, for the rest of the paper we will check all planar groups $G$
and $1\leq r\leq 4$ for ${\it Cay}(G\times R_r, C\times R_r)$ having
genus 1.

\begin{lemm}\label{lem:threereg}
 If the vertex degree of a planar ${\it Cay}(G,C)$ is at least $3$ then ${\it Cay}(G\times R_2,C\times R_2)$
cannot be embedded on the torus.
\end{lemm}
\begin{proof}
Since ${\it Cay}(G,C)$ is at least $3$-regular 
${\it Cay}(G\times R_2,C\times R_2)$ is at least $6$-regular.

Assume that ${\it Cay}(G\times R_2,C\times R_2)$ is embedded on the
torus, then the formula of Euler-Poincar{\'e} yields that all faces are triangular. This implies that every edge of ${\it
Cay}(G\times R_2,C\times R_2)$ lies in at least two triangles, hence
every edge of ${\it Cay}(G,C)$ lies in at least one triangle.

Let $c_1,c_2,c_3\in C$ the generators corresponding to a triangle $a_1,a_2,a_3$. Then $c_1^{\pm 1}c_2^{\pm 1}c_3^{\pm 1}=1_G$ 
for some signing, where $1_G$ is the identity in $G$. If any two of the $c_i$ are
distinct then one of the two is redundant, hence $C$ was not inclusion minimal. Thus every $c\in C$ must
be of order $3$. Since $G$ is not cyclic we obtain that ${\it Cay}(G,C)$ is at least
$4$-regular. The formula of Euler-Poincar{\'e} yields that the at least $8$-regular ${\it Cay}(G\times R_2,C\times R_2)$
cannot be embedded on the torus.
\end{proof}

\begin{figure}[ht]
  \psfrag{a}[cc][cc]{$0$}
  \psfrag{b}[cc][cc]{$1$}
  \psfrag{c}[cc][cc]{$2$}
  \psfrag{d}[cc][cc]{$3$}
  \psfrag{a'}[cc][cc]{$0'$}
  \psfrag{b'}[cc][cc]{$1'$}
  \psfrag{c'}[cc][cc]{$2'$}
  \psfrag{d'}[cc][cc]{$3'$}
  \psfrag{a''}[cc][cc]{$0''$}
  \psfrag{b''}[cc][cc]{$1''$}
  \psfrag{c''}[cc][cc]{$2''$}
  \psfrag{d''}[cc][cc]{$3''$}
\begin{center}
  \includegraphics[width = \textwidth]{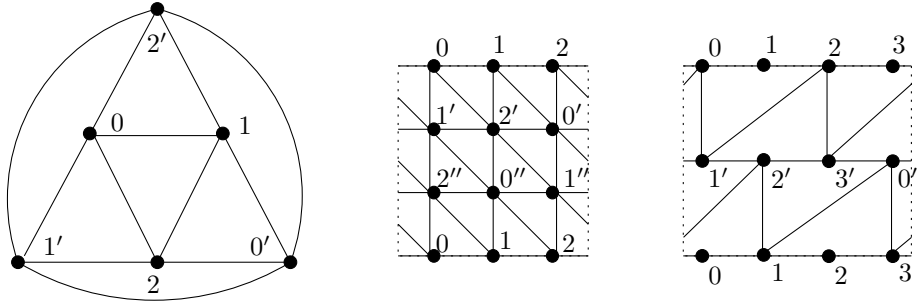}
  \caption{The planar ${\it Cay}(\mathbb{Z}_3\times R_2, \{1\}\times R_2)$, the toroidal ${\it Cay}(\mathbb{Z}_3\times R_3, \{1\}\times R_3)$
and ${{\it Cay}}(\mathbb{Z}_4\times R_2, \{1\}\times R_2)\cong
K_{4,4}$}
  \label{fig:Z34R23}
\end{center}
\end{figure}

\begin{prop}\label{prop:main}
 The minimum genus of ${\it Cay}(\mathbb{Z}_n\times R_r, C\times R_r)$ among all generating systems $C$ is $1$
 iff $(n,r)\in\{(2,3),(2,4),(3,3),(i,2)\}$ for $i\geq 4$.
\end{prop}
\begin{proof}
By Lemma~\ref{lem:threereg} we can assume $C=\{1\}$.

For $n=2$ we have ${\it Cay}(\mathbb{Z}_2\times R_r, C\times
R_r)=K_{r,r}$ which exactly for $r\in\{3,4\}$ has genus 1.

Take $n=3$. If $r=2$ we obtain the planar graph ${\it
Cay}(\mathbb{Z}_3\times R_2, \{1\}\times R_2)$ shown in
Figure~\ref{fig:Z34R23}. If $r=3$ the resulting graph contains
$K_{3,3}$, so it cannot be planar. Figure~\ref{fig:Z34R23} shows an
embedding as a triangular grid on the torus. If $r=4$ we have the
complete tripartite graph $K_{4,4,4}$. Delete the entire set of $16$
edges between two of the partitioning sets. The remaining
(non-planar) graph has $12$ vertices, $32$ edges and, assuming a
toroidal embedding, $20$ faces. A simple count shows that this
cannot be realized without traingular faces. So for $r\geq 4$ the
graph ${\it Cay}(\mathbb{Z}_3\times R_r, C\times R_r)$ is not
toroidal.

Take $n\geq 4$. Now the graph ${\it Cay}(\mathbb{Z}_n,\{1\})$
contains a $C_4=K_{2,2}$ subdivision. If $r\geq 3$ then ${\it
Cay}(\mathbb{Z}_n\times R_r, \{1\}\times R_r)$ is not toroidal by
Proposition~\ref{prop:K22}. If $r=2$ an embedding of ${\it
Cay}(\mathbb{Z}_4\times R_2, \{1\}\times R_2)$ as a square grid in
the torus is shown in Figure~\ref{fig:Z34R23}. This is instructive
for the cases $n\geq 5$. Moreover we see that the vertices
$\{0,0',2\}$ and $\{1,1',3\}$ induce a $K_{3,3}$ subgraph of ${\it
Cay}(\mathbb{Z}_4\times R_2, \{1\}\times R_2)$. Generally for $n\geq
4$ we have that ${\it Cay}(\mathbb{Z}_n\times R_2, \{1\}\times R_2)$
contains a $K_{3,3}$ subdivision, it hence is not planar.
\end{proof}

\begin{theo} \label{maintheo} Let $G\times R_r$ be a finite rightgroup.
The minimal genus of ${\it Cay}(G\times R_r,C\times R_r)$ among all generating sets $C\subseteq G$
of $G$ is $1$ iff $G\times R_r$ is one of the following rightgroups:
\begin{itemize}
 \item $\mathbb{Z}_n\times R_r$ with $(n,r)\in\{(2,3),(2,4),(3,3),(i,2)\}$ for $i\geq4$
\item $\mathbb{Z}_2\times\mathbb{Z}_{2n+1}\times R_2$ for $n\geq 1$
\item $D_n\times R_2$ for all $n\geq 2$
\item $\mathbb{Z}_2\times D_n\times R_2$ for all $n\geq 2$
\end{itemize}
\end{theo}

\begin{proof}

Since $\mathbb{Z}_2\times\mathbb{Z}_{2n+1}\cong \mathbb{Z}_{4n+2}$
Proposition~\ref{prop:main} proves the first two sets of right groups
to have the desired property.

Observe that ${\it Cay}(D_n,C)$, where $C$ consists of two
generators $g_1, g_2$ of order 2, is isomorphic to ${\it
Cay}(\mathbb{Z}_{2n}, \{1\})$. Thus it is planar and by
Proposition~\ref{prop:main} ${\it Cay}(D_n\times
R_2,\{g_1,g_2\}\times R_2)$ can be embedded on the torus. Any other
generating system for $D_n$ yields ${\it Cay}(D_n,C)$ with degree at
least $3$, hence by Lemma~\ref{lem:threereg} it cannot be embedded
on the torus and in particular is non-planar.

The only generating system for $\mathbb{Z}_2\times D_n$ which escapes the
preconditions of Lemma~\ref{lem:threereg} is $C=\{(1,g_1),(0,g_2)\}$ and
indeed ${\it Cay}(\mathbb{Z}_2\times D_n,C)\cong C_{4n}\cong {\it Cay}(\mathbb{Z}_{4n},\{1\})$.
Thus ${\it Cay}(\mathbb{Z}_2\times D_n\times R_2,C\times R_2)$ is toroidal by
Proposition~\ref{prop:main}.

Let $G\in\{A_4,S_4,A_5, \mathbb{Z}_2\times A_4, \mathbb{Z}_2\times
S_4, \mathbb{Z}_2\times A_5, \mathbb{Z}_2\times \mathbb{Z}_{2n}\}$
for $n\geq 2$. It can be checked that $G$ cannot be generated by two
elements of order two. Since $G$ is not cyclic we have $|\{g\in
C\mid \textmd{ord}(g)=2\}|+2|\{g\in C\mid \textmd{ord}(g)\geq 3\}|\geq 3$ for every
generating system $G$. Thus, by Lemma~\ref{lem:threereg} we know
that ${\it Cay}(G\times R_2,C\times R_2)$ cannot be embedded
on the torus.
\end{proof}

In the above proofs we make strong use of Lemma~\ref{lem:threereg},
which tells us that $3$-regular planar Cayley graphs will not be
embeddable on the torus after taking the cartesian product with
$R_2$. In fact, this operation can increase the genus from $0$ to
$3$ already in the following small example.

\begin{figure}[ht]
  \psfrag{0}[cc][cc]{$0$}
  \psfrag{1}[cc][cc]{$1$}
  \psfrag{2}[cc][cc]{$2$}
  \psfrag{3}[cc][cc]{$3$}
  \psfrag{4}[cc][cc]{$4$}
  \psfrag{5}[cc][cc]{$5$}
  \psfrag{0'}[cc][cc]{$0'$}
  \psfrag{1'}[cc][cc]{$1'$}
  \psfrag{2'}[cc][cc]{$2'$}
  \psfrag{3'}[cc][cc]{$3'$}
  \psfrag{4'}[cc][cc]{$4'$}
  \psfrag{5'}[cc][cc]{$5'$}
  \psfrag{A}[cc][cc]{$X$}
  \psfrag{B}[cc][cc]{$Y$}
  \psfrag{C}[cc][cc]{$Z$}
\begin{center}
  \includegraphics[width = .6\textwidth]{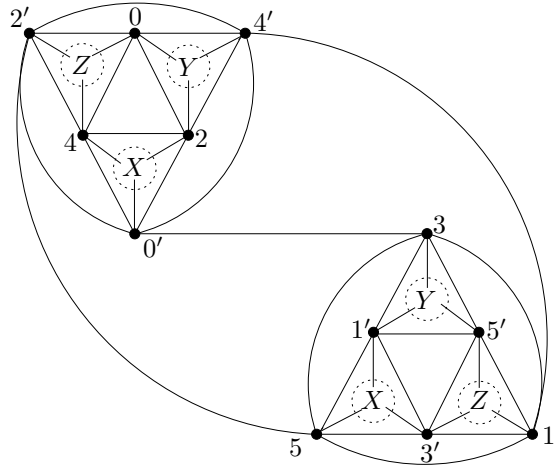}
  \caption{${\it Cay}(\mathbb{Z}_6\times R_2, \{2,3\}\times R_2)$ in the triple torus with handles $X$, $Y$, $Z$.}
  \label{fig:D3R2}
\end{center}
\end{figure}

\begin{exam}\label{prop:D3} The genus of ${\it Cay}(\mathbb{Z}_6\times
R_2, \{2,3\}\times R_2)$ is $3$.
Note that ${\it Cay}(\mathbb{Z}_6\times
R_2, \{2,3\}\times R_2)\cong(C_3\Box K_2)[\overline K_2]$.
\end{exam}
\begin{proof}
To see this we observe that ${\it Cay}(\mathbb{Z}_6\times R_2,
\{2,3\}\times R_2)$ consist of two disjoint copies $C_3\Box K_2$ and
$(C_3\Box K_2)'$ of ${\it Cay}(\mathbb{Z}_6,\{2,3\})$ with vertex
sets $\{0,1,2,3,4,5\}$ and $\{0',1',2',3',4',5'\}$, respectively.
Every vertex $v$ of $C_3\Box K_2$ is adjacent to every neighbor of
its copy $v'$ in $(C_3\Box K_2)'$. Figure~\ref{fig:D3R2} shows an
embedding of ${\it Cay}(\mathbb{Z}_6\times R_2, \{2,3\}\times R_2)$
into the orientable surface of genus $3$ -- \textit{the triple
torus}. This graph is $6$-regular with $12$ vertices, so it has $36$
edges.

By Lemma~\ref{lem:threereg} ${\it Cay}(\mathbb{Z}_6\times R_2,
\{2,3\}\times R_2)$ cannot be embedded on the torus.

So assume that ${\it Cay}(\mathbb{Z}_6\times R_2, \{2,3\}\times
R_2)$ is 2-cell-embedded on the double torus. Delete the $4$ edges
between $1,1'$ and $5,5'$ and the $4$ edges between $0,0'$ and
$4,4'$. The resulting graph $H$ has $28$ edges. It consists of two
graphs
 $A$ and $B$, which are copies of $K_{4,4}$, where $A$
has the bipartition $(\{0,0',5,5'\}$, $\{2,2',3,3'\})$ and $B$ has
$(\{0,0',1,1'\}$, $\{3,3',4,4'\})$. They are glued at the four
vertices with the same numbers and the corresponding $4$ edges are
identified. Although $H$ is no longer bipartite it still is
triangle-free. Hence by our assumption it is  2-cell-embedded on the
double torus. By the formula of Euler-Poincar{\'e} this gives 14
faces and consequently all of them are quadrangular. So the edges
between $1,1'$ and $5,5'$ and between $0,0'$ and $4,4'$, which we
have to put back in, have to be diagonals of these quadrangular
faces. But then $\{2',4,2,0\}$ and $\{2',4,2,0'\}$ are the only
4-cycles in $H$ which contain the vertices $4,0$ and $4,0'$,
respectively, they form faces of $H$. Since they have the common
edges $\{2',4\}$ and $\{2,4\}$ we obtain a $K_{2,3}$ with
bipartition $(\{2,2'\},\{0,0',4\})$. It is folklore that $K_{2,3}$
is not outer planar. Thus the region consisting of the glued
$4$-cycles $\{2',4,2,0\}$ and $\{2',4,2,0'\}$ must contain one of
the vertices $0,0'$ or $4$ in its interior. Hence this vertex has
only degree $2$ -- a contradiction. 
\end{proof}

\noindent We thank Xia Zhang for many helpful comments as well as
Srichan Arworn, Nirutt Pipattanajinda and several graduate students
with whom one of the authors discussed the topic extensively on a
research stay at Chiangmai University, Thailand -- supported by Deutsche
Forschungsgemeinschaft.

\end{document}